\documentclass[12pt,leqno]{article}

\usepackage{amsmath,amsthm,amsfonts,latexsym,amscd,amssymb}
\numberwithin{equation}{section}
\input xypic
\def\qed{{\hbadness=10000\hfill\ \vbox{\hrule height.09ex
   \hbox{\vrule width.09ex height1.55ex depth.2ex \kern1.8ex
   \vrule width.09ex height1.55ex depth.2ex}\hrule height.09ex}\break
   \bigskip}}
\newtheorem{theorem}{Theorem}[section]
\newtheorem{lemma}{Lemma}[section]

\theoremstyle{definition}

\theoremstyle{remark}

\begin{document}

\linespread{1}\title{\textbf{\textsl{h}-exponential change of Finsler metric}}

\author{M.$\,$K. \textsc{Gupta}\thanks{Supported by UGC, Government of India}\, and Anil K.\,\textsc{Gupta}\\
\normalsize{Department of Pure $\&$ Applied Mathematics}\\
\normalsize{Guru Ghasidas Vishwavidyalaya}\\
\normalsize{Bilaspur (C.G.), India}\\
\normalsize{Email: mkgiaps@gmail.com; gupta.anil409@gmail.com}}
\date{}
\maketitle

\linespread{1.3}\begin{abstract} In this paper, we studied a Finsler space whose metric is given by an \textsl{h}-exponential change and obtain the Cartan connection coefficients for the change. We also find the necesssary and sufficient condition for an \textsl{h}-exponential change of Finsler metric to be projective.\\
\noindent\textbf{Keywords:} Finsler space, \textsl{h}-exponential change, projective change.\\
2000 Mathematics Subject Classification:\textbf{ 53B40}.
\end{abstract}

\section{Introduction}
Let $F^n=(M^n,L)$ be an $n$-dimensional Finsler space equipped with the Fundamental function $L(x,y)$. The metric tensor, angular metric tensor and Cartan tensor are defined by $g^{}_{ij}=\frac{1}{2}\dot\partial_i\dot\partial_jL^2$, $h_{ij}=g^{}_{ij}-l_il_j$ \,and\, $C_{ijk}=\frac{1}{2}\dot\partial_ig_{jk}$ respectively, where $\dot\partial_k=\frac{\partial}{\partial y^k}$\,.\,The Cartan connection is given by $C\Gamma = (F^i_{jk},N^i_k,C^i_{jk})$. The \textsl{h}- and \textsl{v}-covariant derivatives $X_{i|j}$ and $X_i|_j$ of a covarient vector field $X_i$ are defined by \cite{mm86,hr59}
\begin{equation} X_{i|j} =\partial_j X_i -N^r_j\,\dot\partial_r X_i-X_r F^r_{ij}\,,\end{equation}\\[-12mm]
and \\[-12mm]
\begin{equation} X_i|_j= \dot\partial_j X_i -X_rC^r_{ij}\,,\end{equation}
where $\partial_k=\frac{\partial}{\partial x^k}$\,. 

\noindent In 2012, H.\,S.\,Shukla et.al.\cite{hb12} considered a Finsler space $ \overline F^n=(M^n,\overline L)$, whose Fundamental metric function is an exponential change of Finsler metric function given by  
\begin{equation}
\overline L=L\,e^\frac{\beta}{L}\,,
\end{equation}
where $\beta=b_i(x)y^i$ is $1$-form on manifold $M^n$.

H. Izumi \cite{hi80} introduced the concept of an \textsl{h}-vector $b_i(x,y)$ which is \textsl{v}-covarient constant with respect to the Cartan connection and satisfies $L\,C^h_{ij}\,b_h=\rho\,h_{ij}$, where $\rho$ is  a non-zero scalar function and $C^i_{jk}$ are components of Cartan tensor.\,Thus if $b_i$ is an \textsl{h}-vector then
\begin{equation}
(i)\, b_i|_k=0,\quad \quad\quad (ii)\,L\,C^h_{ij}b_h=\rho h_{ij}\,.
\end{equation}
From the above definition, we have 
\begin{equation} L\,\dot\partial_j b_i= \rho h_{ij}\,,\end{equation}
which shows that $b_i$ is a function of directional argument also. H.\,Izumi \cite{hi80} proved that  the scalar $\rho$ is independent of directional argument. Gupta and Pandey \cite{mp14} proved that if the \textsl{h}-vector $b_i$ is gradient then the scalar $\rho$ is constant. M.$\,$Matsumoto \cite{ma74} discussed the Cartan connection of Randers change of Finsler metric\,, while B.$\,$N.$\,$Prasad \cite{bn90} obtained the Cartan connection of $(M^n,^*\!\!L)$ where $^*\!L(x,y)$ is given by $^*\!L(x,y)=L(x,y)+b_i(x,y)y^i$, and $b_i(x, y)$ is an \textsl{h}-vector. Gupta and Pandey \cite{mp08,mp09} discussed the hypersurface of a Finsler space whose metric is given by certain transformation with an \textsl{h}-vector. \\$~~~~$ In the present paper, we consider a Finsler space $^*\!F^n=(M^n,{^*}\!L)$, whose metric function  $^*\!L$\,, an \textsl{h}-exponential change of metric\,, is given by
\begin{equation}
^*\!L=L\,e^\frac{\beta}{L}\,,
\end{equation}
where $\beta=b_i(x,y)y^i $ and $b_i$ is an \textsl{h}-vector. And we obtain the relation between Cartan connection coefficients of $F^n$ and $^*\!F^n$. We also derive the condition for an \textsl{h}-exponential change of metric to be projective.

\section { Finsler space $^*\!F^n=(M^n,^*\!\!L)$} 
We shall use following notations  $L_i=\dot\partial_i L=l_i$\,,\,\,\,\,$L_{ij}=\dot\partial_i\dot\partial_j L$\,,\,\,\, $L_{ijk}=\dot\partial_i\dot\partial_j\dot\partial_k L$. The quantities corrosponding to $^*\!F^n$ is denoted by asterisk over that quantity.\\
From (1.6), we have
\begin{equation}^*L_i=e^\tau\big(m_i+l_i\big).\end{equation}
\begin{equation}^*L_{ij}=e^\tau (1+\rho-\tau) L_{ij}+\frac{e^\tau}{L}m_im_j\,.\end{equation}
\begin{equation}
\begin{split}
^*L_{ijk}=&\,e^\tau\big(1+\rho-\tau\big)L_{ijk}+\big(\rho-\tau\big)\frac{e^\tau}{L} \big[m_iL_{jk}+ m_jL_{ik}+ m_kL_{ij}\big]\\& -\frac{e^\tau}{L^2}\big[m_jm_kl_i+ m_im_kl_j+m_im_jl_k-m_im_jm_k\big],
\end{split}
\end{equation}
where $\tau=\frac{\beta}{L}$, $m_i=b_i-\tau l_i$\,. The normalised suporting element, the metric tensor and Cartan tensor of $^*\!F$ are obtained as
\begin{equation}^*l_i=e^\tau\big(m_i+l_i\big),\end{equation}
\begin{equation}
^*\!g_{ij}=\nu\,e^{2\tau}g_{ij} +e^{2\tau}\Big(2\tau^2-\tau-\rho\Big)l_il_j+e^{2\tau}\Big(1-2\tau\Big)\big(b_il_j+b_jl_i\big)+2e^{2\tau}b_ib_j ,
\end{equation}
\begin{equation}
^*\!C_{ijk}=\nu\,e^{2\tau}C_{ijk} +\frac{2}{L}e^{2\tau}m_im_jm_k+\frac{1}{2L}e^{2\tau}(2\nu-1)\big(m_ih_{kj}+m_jh_{ki}+m_kh_{ij}\big),
\end{equation}
where $\nu=1+\rho-\tau$.

For the computation of the inverse metric tensor, we use the following lemma\,\cite{mm72}\,:
\begin{lemma}
Let $(m_{ij})$ be a non-singular matrix and $l_{ij}=m_{ij}+n_in_j$.
The elements $l^{ij}$ of the inverse matrix and determinant of the matrix $(l_{ij})$ are given by
\begin{equation*}l^{ij}=m^{ij}-\big(1+n_kn^k\big)^{-1}n^in^j,\,\,\,\,det\big(l_{ij}\big)=\big(1+n_kn^k\big)det\big(m_{ij}\big) \end{equation*}
respectively, where $m^{ij}$ are elements of inverse matrix $\big(m_{ij}\big)$ and $n^k=m^{ki}n_i$.
\end{lemma}
The inverse metric tencor of $^*F^n$ is derived as follows\,:
\begin{equation}\begin{split}
^*\!g^{ij}=\!\frac{e^{-2\tau}}{\nu}\Big [g^{ij}-\!\frac{1}{m^2+\nu}b^ib^j+\frac{\tau-\nu}{m^2+\nu}\Big(b^il^j+b^jl^i\Big)-\!l^il^j\Big \{\frac{\tau-\nu}{m^2+\!\nu}(m^2+\tau)-\!\rho \Big\} \Big],
\end{split}\end{equation}
where $b$ is magnitude of the vector $b^i=g^{ij}b_j$.\\
From (2.6 ) and (2.7), we obtain
\begin{equation}\begin{split}
^*\!C^h_{ij}= &C^h_{ij}+\frac{1}{m^2+\nu}C_{ijk}b^k (-b^h+2\tau\,l^h-\rho l^h-l^h)\\&+\frac{2}{\nu\,L}\Big[m_im_jm^h+\frac{1}{m^2+\nu}m_im_jm^2(-b^h+2\tau\,l^h-\rho l^h-l^h) \Big]\\&+\frac{1}{2\nu\,L}(2\nu-1)\Big[m_ih^h_j+m_jh^h_i+m^hh_{ij}\\&+\frac{1}{m^2+\nu}(-b^h+2\,\tau\,l^h-\rho l^h-l^h)\,(2m_im_j+m^2h_{ij})\Big].
\end{split}\end{equation}

\section {Cartan connection of the space $^*\!F^n$}
Let $C^*\Gamma=(^*\!F^i_{jk},^*\!\!N^i_j,^*\!C^i_{jk})$ be the Cartan connection for the Finsler space $^*\!F^n=(M^n,^*\!\!L)$.
Since $L_{i|j}=0 $ for the Cartan connection, we have
\begin{equation} \partial_j L_i=L_rF^r_{ij}+\dot\partial_rL_i\,N^r_j\,.\end{equation}
Differentiating (2.1) with respect to $x^j$,\,and using $(1.1)$ and $(3.1)$, we get
\begin{equation}\begin{split}
{^*\!L}_{ir}{^*N}^r_j+{^*}\!L_r{^*}\!F^r_{ij}
=&\Big[e^\tau\,\nu\,L_{ir}+\frac{e^\tau}{L}m_r m_i\Big] N^r_j +\Big[e^\tau\big(m_r+l_r\big)\Big]F^r_{ij}+\frac{e^\tau \beta_{j}\,m_i}{L}+e^\tau b_{i|j}\,.
\end{split}\end{equation}

\noindent Equation (3.2) serves the purpose to find relation between cartan connection of $^*\!F^n$ and $F^n$. For this, we put
\begin{equation} D^i_{jk}={^*}\!F^i_{jk}-F^i_{jk}\,.\end{equation}
With the help of (3.3), the equation (3.2) becomes
\begin{equation}
\Big[e^\tau\,\nu\,L_{ir}+\frac{e^\tau}{L}m_im_r\Big]D^r_{0j}+\Big[e^\tau\big(m_r+l_r\big)\Big]D^r_{ij} = \frac{e^\tau\beta_{|j}m_i}{L}+e^\tau b_{i|j}\,,
\end{equation}
where the subscrit `0' denote the contraction by $y^i$.\\
Differentiating $(2.2)$ with respect to $x^k$,\, and using (1.1) and (3.1), we have
\begin{equation}\begin{split}
& e^\tau\,\nu\Big [L_{ijr}D^r_{0k}+L_{rj}D^r_{ik}+L_{ir}D^r_{jk}\Big ]+\big(\nu-1\big)\frac{e^\tau}{L}\Big[m_rL_{ij}+m_iL_{jr}+m_jL_{ir}\Big]D^r_{0k}\\&-\frac{e^\tau}{L^2}\Big[m_im_jl_r+m_jm_rl_i+m_rm_il_r-m_im_jm_r\Big]D^r_{0k}+ \frac{e^\tau}{L}\Big[m_rm_jD^r_{ik}+m_im_rD^r_{jk}\Big]\\&-\frac{e^\tau\big(\nu-1\big)}{L}L_{ij}\,\beta_{|k}-\frac{e^\tau}{L^2}\beta_{|k}m_im_j-e^\tau\rho_kL_{ij}=0\,,
\end{split}\end{equation}
where $\rho^{}_k=\rho_{|k}=\partial_k\rho$\,.
\begin{theorem}
The Cartan connection of $^*\!F^n$ is completely determine by the equations (3.4) and (3.5) .
\end{theorem}
To prove this, first  we propose the following lemma :
\begin{lemma}
System of equations 
\begin{center}
$~~~~$$(i)\quad {^*\!L_{ir}\,A^r=B_i}$ \\
$~~~~$$(ii)\quad {^*\!L_r\,A^r=B}$\\
\end{center}
 has unique solution $A^r$ for given $ B $ and $B_i$.
\end{lemma}
\begin{proof}
Using (2.2), equation (\textit{i}) becomes
\begin{equation}\begin{split}
\frac{e^\tau}{L}\Big[\nu\big({g_{ir}-l_il_r}\big)+m_im_r\Big]A^r=B_i \,.
\end{split}\end{equation}
Contracting by $b^i$, we get
\begin{equation}\begin{split}
m_rA^r=\frac{LB_\beta}{e^\tau}\big(m^2+\nu\big)^{-1}\,,
\end{split}\end{equation}
here we used subscript $\beta$ to denote the contraction by $b^i$,\,\,i.e. $B_{\beta}=B_ib^i$.\\
From (2.1) and \textit{(ii)}\,, we have
\begin{equation}
l_rA_r=\frac{B}{e^\tau}-\frac{LB_\beta}{e^\tau}\big(m^2+\nu\big)^{-1}\,.
\end{equation}
Using $(3.7)$ and $(3.8)$,\, equation$(3.6)$ becomes
\begin{equation*}\begin{split}
 g_{ir}A^r=\frac{LB_i}{\nu\,e^\tau}+ l_i\Big[\frac{B}{e^\tau}-\frac{LB_\beta}{e^\tau}\big(m^2+\nu\big)^{-1}\Big]-\frac{m_iLB_\beta}{\nu\,e^\tau}\big(m^2+\nu\big)^{-1}\,,\end{split}\end{equation*}
contracting by $g^{ij}$\,, we have  
\begin{equation}\begin{split}
 A^j=\frac{LB^j}{\nu\,e^\tau}+l^j\Big[\frac{B}{e^\tau}-\frac{LB_\beta}{e^\tau}\big(m^2+\nu\big)^{-1}\Big]-\frac{m^jLB_\beta}{\nu\,e^\tau}\big(m^2+\nu\big)^{-1}\,, 
 \end{split}\end{equation}
which is concrete form of the solution $A^j$.
\end{proof}

Now we are in the position to prove the theorem. We will find an explicit expression of difference tensor $D^i_{jk}$ in three steps. Firstly, we will find $D^i_{00}$ and then $D^i_{0k}$ and in the last $D^i_{jk}$.

Taking symmetric and skew-symmetric part of $(3.4)$, we have 
\begin{equation}
\begin{split}
2e^\tau\big(m_r+l_r\big)D^r_{ij}+\Big[\nu\,e^\tau L_{ir}+\frac{e^\tau}{L}m_im_r\Big]D^r_{0j}+&\Big[\nu\,e^\tau L_{jr}+\frac{e^\tau}{L}m_jm_r\Big]D^r_{0i} \\&=\frac{e^\tau}{L}\big(\beta_{|j}m_i+\beta_{|i}m_j\big)+2e^\tau E_{ij}\,,\end{split}
\end{equation}
and
\begin{equation}\begin{split}
\Big[\nu\,e^\tau L_{ir}+\frac{e^\tau}{L}m_im_r\Big]D^r_{0j}-&\,\Big[\nu\,e^\tau L_{jr}+\frac{e^\tau}{L}m_jm_r\Big]D^r_{0i} \\&=\frac{e^\tau}{L}\big(\beta_{|j}m_i-\beta_{|i}m_j\big)+2e^\tau F_{ij}, 
\end{split}\end{equation}
where $2E_{ij}=b_{j|i}+b_{i|j}, ~~~2F_{ij}=b_{i|j}-b_{j|i}.$\\
Contracting (3.10) and (3.11) by $y^j$, we get 
\begin{equation}
2e^\tau\big(m_r+l_r\big)D^r_{0i}+\Big[\nu\,e^\tau L_{ir}+\frac{e^\tau}{L}m_im_r\Big]D^r_{00} =\frac{e^\tau}{L}\beta_{|0}m_i+2e^\tau E_{i0} \,,
\end{equation}
and
\begin{equation}
\Big[\nu\,e^\tau L_{ir}+\frac{e^\tau}{L}m_im_r\Big]D^r_{00}=\frac{e^\tau}{L}\beta_{|0} m_i+2e^\tau F_{i0}\,,
\end{equation}
which may be re-written as
\begin{equation}
^*L_{ir}D^r_{00}=\frac{e^\tau}{L}\beta_{|0} m_i+2e^\tau F_{i0}\,,
\end{equation}
where $\beta_{|0}=\beta_{|j}y^j$. Transvecting  (3.13) by $ m^i$, we obtain
\begin{equation}
m_r D^r_{00}= \big(m^2+\nu\big)^{-1}\big(\beta_{|0} m^2+2LF_{\beta0}\big)\,.
\end{equation}
Contracting (3.12) by $y^i$,\,we get
\begin{equation*}
2e^\tau\big(m_r+l_r\big)D^r_{00}=2e^\tau E_{00}.
\end{equation*}\\[-13mm] 
\textit {i.e.}\\[-12mm]
\begin{equation}
^*L_rD^r_{00}=e^\tau E_{00}\,.
\end{equation}
Applying Lemma 3.1 in equation (3.14) and (3.16)\,, we have
\begin{equation}\begin{split}
D^i_{00}=&\frac{L}{\nu\,e^\tau}\Big[\frac{e^\tau}{L}\beta_{|0}m^i+2e^\tau F^i_0\Big]+l^i\Big[E_{00}-\frac{L}{e^\tau}\big(m^2+\nu\big)^{-1}\big(\frac{e^\tau}{L}\beta_{|0}m^2+2e^\tau F_{\beta 0}\big)\Big]\\
& -\frac{m^iL}{\nu\,e^\tau}\big(m^2+\nu\big)^{-1}\Big[\frac{e^\tau}{L}\beta_{|0}m^2+2e^\tau F_{\beta 0}\Big]\,.
\end{split}\end{equation}
Here we used $m^ib_i=m_im^i=m^2$. Also we note that $E_{00}=E_{ij}y^iy^j=b_{i|j}y^iy^j=(b_iy^i)_{|j}y^j=\beta_{|0}$,$~$$F^i_0=g^{ij}F_{j0}$.

Secondly, applying Christoffel process with respect to indices $i,j,k$ in equation (3.5), we have
\begin{equation}\begin{split}
& \nu\,e^\tau\Big[L_{ij r}D^r_{0k}+L_{jkr}D^r_{0i}-L_{kir}D^r_{0j}\Big] +2D^r_{ik}\Big[\nu\,e^\tau\,L_{jr}+\frac{e^\tau}{L}m_rm_j\Big]\\
&+\frac{e^\tau}{L}D^r_{0k}\mathfrak{S}_{(rij)}\Big[(\nu-\!1) m_rL_{ij}\!-\!\frac{m_im_jl_r}{L}\Big]
\!+\!\frac{e^\tau}{L}D^r_{0i}\mathfrak{S}_{(rjk)}\Big[(\nu-\!1) m_rL_{jk}-\!\frac{m_jm_kl_r}{L}\Big] \\
&-\frac{e^\tau}{L}D^r_{0j}\mathfrak{S}_{(rki)}\Big[(\nu-1) m_rL_{ki}-\frac{m_km_il_r}{L}\Big] -e^\tau  \Big[\rho_kL_{ij}+\rho_iL_{jk}-\rho_j L_{ki}\Big]\\
&-\big(\nu-1\big)\frac{e^\tau}{L}\big(\beta_{|k}L_{ij}+\beta_{|i}L_{jk}-\beta_{|j}L_{ki}\big)
-\frac{e^\tau}{L^2}\Big[ \beta_{|k}m_im_j+\beta_{|i}m_jm_k-\beta_{|j}m_km_i\Big] \\
&+\frac{e^\tau}{L^2}\Big[m_im_jm_rD^r_{0k}+m_jm_km_rD^r_{0i}-m_km_im_rD^r_{0j}\Big]=0\,,
\end{split}\end{equation}
where $\mathfrak{S}_{(ijk)}$ denote cyclic interchange of indices $i$,$j$,$k$ and summation. 
Contracting by $y^k$, above equation becomes 
\begin{equation}\begin{split}
&\nu\,e^\tau\Big[L_{ijr}D^r_{00}-L_{jr}D^r_{0i}+L_{ir}D^r_{0j}\Big]+2D^r_{0i}\,\Big[\nu\,e^\tau\,L_{jr}+\frac{e^\tau}{L}m_rm_j\Big] \\
& +\frac{e^\tau}{L}D^r_{00}\mathfrak{S}_{(rij)}\Big[(\nu-1) m_rL_{ij}-\frac{m_im_jl_r}{L}\Big]-\frac{e^\tau}{L}D^r_{0i}\frac{m_rm_jl_k}{L}y^k \\&
 +\frac{e^\tau}{L}D^r_{0j}\frac{m_rm_il_k}{L}y^k +\frac{e^\tau}{L^2}m_im_jm_rD^r_{00}-(\nu-1)\frac{e^\tau}{L}\beta_{|0}L_{ij}\\& 
-\frac{e^\tau}{L^2} \beta_{|0}m_im_j -e^\tau  \rho_0L_{ij}=0\,.
\end{split}
\end{equation}
Adding (3.11) and (3.19)\,, we have 
\begin{equation}
^*L_{ir}D^r_{0j}=G_{ij}\,,
\end{equation}\\[-13 mm]
where\\[-6 mm]
\begin{equation}\begin{split}
 2G_{ij}=&\frac{e^\tau}{L}\big(\beta_{|j}m_i-\beta_{|i}m_j\big)-\!e^\tau \nu\,L_{ijr}D^r_{00}-\!\frac{e^\tau}{L}D^r_{00}\mathfrak{S}_{(rij)}\Big[(\nu-\!1)m_rL_{ij}-\!\frac{m_im_jm_r}{L}\Big]\\
&+2e^\tau F_{ij}-\frac {e^\tau}{L^2}m_rm_im_jD^r_{00}+\frac{(\nu-\!1)}{L}e^\tau \beta_{|0}L_{ij}+\frac{e^\tau}{L^2}B_0m_im_j+e^\tau \rho_0 L_{ij}\,.
\end{split}\end{equation}
Equation (3.12) can be written as
\begin{equation}\begin{split}
^*L_rD^r_{0j} & = G_j\,,
\end{split}\end{equation}
where
\begin{equation*}
2G_j=\frac{e^\tau}{L}\beta_{|0}m_j+2e^\tau E_{j0}+\Big[-{e^\tau \nu\,L_{jr}}-\frac{e^\tau m_jm_r}{L}\Big]D^r_{00}\,.
\end{equation*}
Using (3.13), above equation may be written as
\begin{equation}
G_j=e^\tau\big(E_{j0}-F_{j0}\big).
\end{equation}
Applying Lemma 3.1\, in equation $(3.20)$ and $(3.22)$\,,  we obtain
\begin{equation}
D^i_{0j}=\frac{LG^i_j}{\nu\,e^\tau}+\frac{l^i}{e^\tau}\Big[G_j\,-LG_{\beta j}\big(m^2+\nu\big)^{-1}\Big]-\frac{m^iLG_{\beta j}}{\nu\,e^\tau}\big(m^2+\nu\big)^{-1}\,.
\end{equation}
\indent Finally,\,the equation (3.10) may be written as
\begin{equation}
{^*}L_rD^r_{ik}=H_{ik}\,,\\[-5mm]
\end{equation}
where 
\begin{equation}\begin{split}
2H_{ik}=\frac{e^\tau}{L}\big(\beta_{|k}m_i+\beta_{|i}m_k\big)&+e^\tau E_{ik}-\Big[e^\tau \nu\,L_{ir}+\frac{e^\tau}{L}m_im_r\Big]D^r_{0k}\\
&-\Big[\,e^\tau \nu\,L_{kr}+\frac{e^\tau}{L}m_km_r\Big]D^r_{0i}\,.
\end{split}\end{equation}
Equation (3.18) may be written as
\begin{equation}
^*L_{rj}D^r_{ik}=H_{jik}\,,\\[-5mm]
\end{equation}
where\\[-5mm]
\begin{equation}\begin{split}
2H_{jik}=& -\nu\,e^\tau\Big[L_{ij r}D^r_{0k}+L_{jkr}D^r_{0i}-L_{kir}D^r_{0j}\Big]+e^\tau  \Big[\rho_kL_{ij}+\rho_iL_{jk}-\rho_j L_{ki}\Big] \\
&-\frac{e^\tau}{L}D^r_{0k}\mathfrak{S}_{(rij)}\Big[(\nu-\!1) m_rL_{ij}-\frac{m_im_jl_r}{L}\Big]-\frac{e^\tau}{L}D^r_{0i}\mathfrak{S}_{(rjk)}\Big[(\nu-\!1) m_rL_{jk}-\frac{m_jm_kl_r}{L}\Big] \\
&+\frac{e^\tau}{L}D^r_{0j}\mathfrak{S}_{(rki)}\Big[(\nu-\!1) m_rL_{ki}-\frac{m_km_il_r}{L}\Big] \\
&+(\nu-\!1)\frac{e^\tau}{L}\big(\beta_{|k}L_{ij}+\beta_{|i}L_{jk}-\beta_{|j}L_{ki}\big)
+\frac{e^\tau}{L^2}\Big[ \beta_{|k}m_im_j+\beta_{|i}m_jm_k-\beta_{|j}m_km_i\Big] \\
&-\frac{e^\tau}{L^2}\Big[m_im_jm_rD^r_{0k}+m_jm_km_rD^r_{0i}-m_km_im_rD^r_{0j}\Big]\,.
\end{split}\end{equation}

\noindent Applying Lemma 3.1\, in $(3.25)$ and $(3.27)$, we have
\begin{equation}\begin{split}
D^j_{ik}=\frac{LH^j_{ik}}{\nu\,e^\tau}+& \frac{l^j}{e^\tau}\Big[{H_{ik}}-{LH_{\beta ik}\big(m^2+\nu\big)^{-1}}\Big]-\frac{m^jL}{\nu\,e^\tau} H_{\beta ik}\big(m^2+\nu\big)^{-1} ,
\end{split}\end{equation}
where we put $H^j_{ik}=g^{jm}H_{mik}$.\\
Thus in view of (3.3), we get the Cartan connection coefficient $^*\!F^i_{jk}$\,. This completes the proof of theorem (3.1).\\

\indent Now, suppose Cartan connection coefficients for both spaces $F^n$ and $^*\!F^n$ are same\,, \textit {i.e.} $^*\!F^i_{jk}$=$F^i_{jk}$. Then $D^i_{jk}=0$.\,But then equations (3.12) and (3.13) implies that $E_{i0}=F_{i0}$\,,\,\,and hence 
\begin{equation} b_{0|i}=0\,,\end{equation}
\textit {i.e.} $\beta_{|i}=0$. Differentiating $\beta_{|i}=0$ partially with respect to $y^j$ and applying  commutation formulae $\dot\partial_j(\beta_{|i})-(\dot\partial_j\beta)_{|i}=-(\dot\partial_r \beta)C^r_{ij|0}$\,, we get
\begin{equation}
b_{j|i}=b_rC^r_{ij|0}\,.
\end{equation}
From the above equation, we conclude that $F_{ij}=0$. M. K. Gupta and P. N. Pandey \cite{mp14} has proved that if \textsl{h}-vector $b_i$ is gradient, i.e. $F_{ij}=0$ then $\rho$ is constant, i.e. $\rho_i=\rho_{|i}=0$. Taking \textsl{h}-covariant derivative  of $LC^r_{ij}b_r=\rho h_{ij}$ and using \,$L_{|k}=0$,\,\,$\rho_{|k}=0$\, and $h_{ij|k}=0$,\, we have
\begin{equation*}
 (b_rC^r_{ij})_{k}=\frac{\rho}{L}h_{ij}=0\,,
\end{equation*}
\textit {i.e.}, 
\begin{equation*}
b_{r|k}C^r_{ij}+b_rC^r_{ij|k}=0\,.
\end{equation*}
From (3.31)\,, $b_{r|k}=b_{k|r}$ and hence above equation becomes 
\begin{equation*} b_{k|r}C^r_{ij}+b_rC^r_{ij|k}=0\,. \end{equation*}
Transvecting by $y^k$, we have $b_{0|r}C^r_{ij}+b_rC^r_{ij|0}=0$.\,Using (3.30) and (3.31), we conclude that $b_{i|j}=0 $.\\$~~~~$ Conversely, $b_{i|j}=0$ implies that $E_{ij}=0=F_{ij}$ and $\beta_{|i}=\beta_{|i}=b_{j|i}=0$.\,$F_{ij}=0$\, implies that $\rho_i=\rho_{|i}=0$\cite{mp14}. Therefore from $(3.17)$, we get $D^i_{00}=0$ and then $G_{ij}=0$ and $G_j=0$. This gives $D^i_{0j}=0$ and then $H_{jik}=0$ and $H_{ik}=0$. Therefore (3.29) implies that $D^i_{jk}=0$. Thus\,, we have\,:
\begin{theorem}
For an \textsl{h}-exponential change of metric\,, the Cartan connection coefficients for both spaces $F^n$ and $^*\!F^n$ are same if and only if the \textsl{h}-vector $b_i$ is parallel with respect to Cartan connection of $F^n$.
\end{theorem}
Now transvecting (3.3) by $y^j$ and using $F^i_{jk}\,y^j=G^i_k$, we obtain
\begin{equation}
^*G^i_k=G^i_k+D^i_{0k}\,.
\end{equation}
Transvecting again the above equation by $y^k$ and using $G^i_k y^k=2G^i$, we get
\begin{equation}
2\,^*\!G^i=2G^i+D^i_{00}\,.
\end{equation}
Differentiating (3.32) partially with respect to $y^h$ and using $\dot\partial_h G^i_k=G^i_{kh}$\,, we have
\begin{equation}
^*G^i_{kh}=G^i_{kh}+\dot\partial_h D^i_{0k}\,,
\end{equation}
where $G^i_{kh}$ are Berwald connection coefficients.\\
$~~~~$ Now, if the \textsl{h}-vector $b_i$ is parallel with respect to Cartan connection of $F^n$ then by Theorem (3.2), the Cartan connection coefficients for both spaces $F^n$ and ${^*}\!F^n$ are same, therefore $D^i_{jk}=0$. Hence from (3.34), we get $^*G^i_{kh}=G^i_{kh}$.\\Thus, we have\,:
\begin{theorem}
For an \textsl{h}-exponential change of metric\,, if an \textsl{h}-vector $b_i$ is parallel with respect to Cartan connection of $F^n$ then Barwald connection coefficients for both spaces $F^n$ and $^*\!F^n$ are the same.
\end{theorem}

\section {Condition for \textsl{h}-exponential change of metric to be projective}
Let us consider Finsler spaces $F^n=(M^n,L)$ and $^*\!F^n=(M^n, ^*\!L)$. A transformation from $L$ to $^*\!L$ is called projective change if any geodesis on $F^n=(M^n,L)$ is also geodesis on $^*\!F^n=(M^n, ^*\!\!L)$ and vice-versa. A geodesis on $F^n$ is given by
\begin{equation*}
\frac{dy^i}{dx} +2G^i(x,y)=\tau y^i\,; \quad \tau=\frac{d^2s/dt^2}{ds/dt}
\end{equation*}
The change $L\mapsto {^*}\!L$ is projective change if and only if there exits a scaler function $P(x,y)$ which is positive homogeneous of degree one in $y^i$ and satisfies
\begin{equation*}
^*G (x,y)=G^i(x,y)+ P(x,y) y^i \,.
\end{equation*}
Now, we find the condition for exponential change with \textsl{h}-vector to be projective. From (3.33), it follows that exponential change with \textsl{h}-vector to be projective if and only if $D^i_{00}=2Py^i$. Then from (3.17), we get
\begin{equation}\begin{split}
2Py^i=&\frac{L}{\nu\,e^\tau}\Big[\frac{e^\tau}{L}\beta_{|0}m^i+2e^\tau F^i_0\Big]+l^i\Big[E_{00}-\frac{L}{e^\tau }\big(m^2+\nu\big)^{-1}\big(\frac{e^\tau}{L}\beta_{|0}m^2+2e^\tau F_{\beta 0}\big)\Big]\\
& -\frac{m^iL}{\nu\,e^\tau}\big(m^2+\nu\big)^{-1}\Big[\frac{e^\tau}{L}\beta_{|0}m^2+2e^\tau F_{\beta 0}\Big]\,.
\end{split}\end{equation} 
Transvecting  (4.1) by $y_i$ and using $m^iy_i=0$, $F^i_0 y_i=0$, we get
\begin{equation}
P=\frac{y_il^i}{2L^2}\Big[E_{00}-\frac{L}{e^\tau }\big(m^2+\nu\big)^{-1}\big(\frac{e^\tau}{L}\beta_{|0}m^2+2e^\tau F_{\beta 0}\big)\Big].
\end{equation}
Substituting the value of $P$ in(4.1), we get
\begin{equation}
F^i_0=\frac{m^i}{2L}\big(m^2+\nu\big)^{-1}\big(\beta_{|0}m^2+2L F_{\beta 0}\big)-\frac{\beta_{|0} m^i}{2L}\,.
\end{equation}
Using (3.15) in above equation, we have
\begin{equation}
F^i_0=\frac{m^i}{2L} m_rD^r_{00}-\frac{\beta_{|0} m^i}{2L}\,.
\end{equation}
Transvecting by $g_{ij}$ to above equation, we have
\begin{equation}
F_{i0}=\frac{m_i}{2L} m_rD^r_{00}-\frac{\beta_{|0} m_i}{2L}\,.
\end{equation}
Using (4.5) in (3.13) and reffering $\nu\neq 0$, we obtain $L_{ir}D^r_{00}=0$, which transvecting by $m^i$ and using $L_{ir}m^i=\frac{1}{L}m_r$, we get $m_rD^r_{00}=0$ ,\,and then (4.5) becomes
\begin{equation}
F_{i0}=-\frac{\beta_{|0} m_i}{2L}\,.
\end{equation}
The equation (4.6) is necessary condition for \textsl{h}-exponential change  to be projective change.\\
Conversely, if (4.6) satisfied\,, the equation (3.13) yields
\begin{equation}
\Big[e^\tau \nu\,L_{ir}+\frac{e^\tau}{L}m_im_r\Big]D^r_{00}=0\,.
\end{equation}
Transvecting by $m^i$ and referring $(m^2+\nu) \neq 0,$ we get $m_rD^r_{00}=0$ and then (3.17) gives $D^i_{00}=E_{00}l^i$. Therefore $^*\!F^n$ is projective to $F^n$. Thus\,, we have\,: 
\begin{theorem}
The \textsl{h}-exponential change given by (1.6) is projective if and only if condition (4.6)  is satisfied.
\end{theorem}

\end{document}